\documentclass[a4paper]{amsart}
\usepackage{amssymb}
\usepackage{amscd}
\usepackage{amsthm}
\usepackage{amsmath}
\usepackage{mathtools}
\usepackage{latexsym}
\usepackage{hyperref}
\usepackage[all]{xy}
\usepackage[utf8]{inputenc}
\usepackage{enumitem}

\usepackage{tikz-cd}

\setlist[enumerate]{label={(\roman*)}}

\theoremstyle{plain}
\newtheorem{theorem}{Theorem}

\newtheorem{lemma}[theorem]{Lemma}
\newtheorem{proposition}[theorem]{Proposition}

\theoremstyle{definition}
\newtheorem{definition}[theorem]{Definition}
\newtheorem{example}[theorem]{Example}

\theoremstyle{remark}

\newtheorem{remark}{Remark}
\newtheorem*{acknowledgment}{Acknowledgement}
\numberwithin{theorem}{section}
\usepackage{newtxtext}
\usepackage{newtxmath}

\DeclareMathAlphabet\urwscr{U}{urwchancal}{m}{n}%
\DeclareMathAlphabet\rsfscr{U}{rsfso}{m}{n}
\DeclareMathAlphabet\euscr{U}{eus}{m}{n}
\DeclareFontEncoding{LS2}{}{}
\DeclareFontSubstitution{LS2}{stix}{m}{n}
\DeclareMathAlphabet\stixcal{LS2}{stixcal}{m} {n}
\newcommand{\Filt}[1]{\mbox{\rm Filt}(#1)}

\newcommand{\card}[1]{|#1|}
\newcommand{\cf}[1]{\mbox{\rm cf}(#1)}

\newcommand{\im}{\mbox{\rm{Im\,}}}

\newcommand{\Ext}[4]{\operatorname{Ext}^{#1}_{#2}(#3,#4)}

\newcommand{\rmod}[1]{\mbox{\rm{Mod}--}{#1}}

\newcommand{\LS}{\mbox{\rm{LS}}}

\begin{document}

\title{Deconstructible abstract elementary classes of modules and categoricity}

\author{Jan \v Saroch and Jan Trlifaj}
\address{Charles University, Faculty of Mathematics
and Physics, Department of Algebra \\
Sokolovsk\'{a}~83, 186 75 Prague 8, Czech Republic}
\email{saroch@karlin.mff.cuni.cz}
\email{trlifaj@karlin.mff.cuni.cz}

\begin{abstract} We prove a version of Shelah's Categoricity Conjecture for arbitrary deconstructible classes of modules. Moreover, we show that if $\mathcal A$ is a deconstructible class of modules that fits in an abstract elementary class $(\mathcal A,\preceq)$ such that (1) $\mathcal A$ is closed under direct summands and (2) $\preceq$ refines direct summands, then $\mathcal A$ is closed under arbitrary direct limits. In an Appendix, we prove that the assumption (2) is not needed in some models of ZFC.  
\end{abstract}

\date{\today}

\thanks{The authors acknowledge the support by GA\v CR grant no.\ 23-05148S, and by the Fields Institute within the Thematic Program on Set Theoretic Methods in Algebra, Dynamics and Geometry.}
	


\maketitle

\section{Introduction}

The classic fact that each projective module decomposes into a direct sum of countably generated modules proved by Kaplansky, and the Matlis Decomposition Theorem for injective modules over commutative noetherian rings, are among the key tools of the structure theory of modules. However, decomposability is rare for general classes of (not necessarily finitely generated) modules. The weaker property of deconstructibility is much more frequent. 

A class of modules $\mathcal D$ is \emph{deconstructible}, if $\mathcal D$ contains a subset $\mathcal C$ such that $\mathcal D$ is the class of all transfinite extensions of modules from $\mathcal C$. If $\mathcal D$ is a deconstructible class of modules, then for each $n \geq 0$, the class $\mathcal D_n$ of all modules of $\mathcal D$-resolution dimension $\leq n$ is also deconstructible \cite{ST}. In particular, so is the class of all modules of projective dimension $\leq n$ over any ring \cite{AEJO}. 

From the point of view of homological algebra, the key property of deconstructible classes is that they provide precovers (or right approximations) \cite{SS}, making it possible to develop relative homological algebra \cite{EJ}. For example, Enochs' proof of the Flat Cover Conjecture in \cite{BEE} is based on the deconstructibility of the class of all flat modules over any ring. 

\medskip
In \cite{BET}, particular deconstructible classes of modules, $\mathcal A _e$, were used to construct interesting abstract elementary classes (AECs) of modules, $(\mathcal A _e, \preceq _e)$, called the AECs of roots of Ext. These AECs provided a link between representation theory and homological algebra on the one hand, and abstract model theory on the other hand. It turned out that all the AECs of roots of Ext have the additional property of being closed under arbitrary direct limits, hence they provide covers of modules. However, the question about the number of models of various cardinalities in $\mathcal A _e$ remained open, cf.\ \cite[Question 4.1]{BET}.    

\medskip
In this paper, we are interested in general deconstructible classes of modules, $\mathcal A$, and in particular in those which equipped with an appropriate relation of a strong submodule form a general abstract elementary class (AEC), $(\mathcal A, \preceq)$, in the sense of \cite{B} and \cite{S}.  

Our first main result, Theorem \ref{t:closedl}, says that the closure under direct limits, proved in the particular case of the deconstructible classes $\mathcal A _e$ in \cite{BET}, is a much more general phenomenon. We show that if $\mathcal A$ is any deconstructible class of modules closed under direct summands which fits in an AEC $(\mathcal A, \preceq)$ such that the partial order $\preceq$ refines direct summands, then $\mathcal A$ is closed under arbitrary direct limits. Moreover, in the Appendix, we show that the assumption that $\preceq$ refines direct summands is actually not needed in some models of ZFC.
  
\medskip
Shelah Categoricity Conjecture (SCC) is a key open problem concerning the number of models for AECs consisting of general structures, see e.g., \cite[N.4.3]{S}, \cite{ABV}, and \cite{V}. In the particular setting of AECs of roots of Ext, SCC was proved in \cite{T}. Our second main result, Theorem \ref{maindec}, proves a version of SCC for arbitrary deconstructible classes of modules.

\section{Preliminaries}

\medskip
\subsection{Abstract elementary classes}

We begin by recalling the general notion of an AEC in the setting of the first-order theories of modules.  

In what follows, $R$ will denote an (associative, unital) ring, $\rmod R$ the category of all (unitary right $R$-) modules, and $\mathcal A$ a class of modules. Moreover, $\preceq$ will denote a partial order on $\mathcal A$ such that for all $A, B \in \mathcal A$, if $A \preceq B$, then $A$ is a submodule of $B$. In what follows, we will require both $\mathcal A$ and $\preceq$ to be closed under isomorphisms (in the latter case, this means that $N^\prime \preceq M^\prime$, whenever $N \preceq M$ and there is an isomorphism $f\colon M \to M^\prime$ such that $f(N) = N^\prime$). We will say that $\preceq$ \emph{refines direct summands} provided that $A \preceq B$ whenever $A, B \in \mathcal A$ and $A$ is a direct summand in $B$. 

If $A, B \in \mathcal A$ and $A \preceq B$, then $A$ is called a \emph{strong submodule} of $B$. A monomorphism $f\colon A \to B$ with $A, B \in \mathcal A$ is \emph{strong} in case $f(A) \preceq B$. A module $M$ is \emph{strongly presented} in case there exists a short exact sequence $0 \to A \overset{f}\to B \to  M \to 0$ such that $A, B \in \mathcal A$ and $f$ is strong.  
 
If $\delta$ is an ordinal, then a chain of modules $( A_i \mid i < \delta )$ with $A_i \in \mathcal A$ for all $i < \delta$ is \emph{continuous} provided that $A_i = \bigcup_{j<i} A_j$ for each limit ordinal $i < \delta$, and it is a \emph{continuous $\preceq$-increasing chain} in $\mathcal A$ if moreover $A_i$ is a strong submodule of $A_{i+1}$ for each $i+1 < \delta$.  

\begin{definition}\label{defroots}
A pair $\mathbb A = (\mathcal A, \preceq )$ is an \emph{abstract elementary class} (or an \emph{AEC}) of modules, in case the following conditions are satisfied:
\begin{itemize}
\item[{\rm{(A1)}}] If $( A_i \mid i < \delta )$ is a continuous $\preceq$-increasing chain in $\mathcal A$, then
\begin{enumerate}
\item $\bigcup_{i<\delta} A_i \in \mathcal A$,
\item $A_j \preceq \bigcup_{i<\delta}A_i$ for each $j<\delta$, and
\item if $M \in \mathcal A$ and $A_i \preceq  M$ for each $i<\delta$, then $\bigcup_{i<\delta}A_i \preceq M$.
\end{enumerate}
\item[{\rm{(A2)}}] If $A,B,C \in \mathcal A$, $A \preceq C$, $B \preceq C$, and $A$ is a submodule of $B$, then $A \preceq B$.
\item[{\rm{(A3)}}] There exists a cardinal $\lambda$ such that, if $A$ is a submodule of a module $B \in \mathcal A$, then there exists
$A^\prime \in \mathcal A$ such that $A \subseteq A^\prime \preceq B$, and $|A^\prime| \leq |A| + \lambda$. The least such infinite cardinal $\lambda \geq \card R$ is called the \emph{L\" owenheim--Skolem number} of $\mathbb A$, and denoted by $\LS(\mathbb A)$.
\end{itemize}

\medskip
An AEC \emph{of the roots of Ext} is an abstract elementary class of modules of the form $(\mathcal A, \preceq _e)$ where 

\begin{itemize}
\item $\mathcal A$ is the \emph{left-hand class of a hereditary cotorsion pair}, that is, 
$$\mathcal A  = {}^{\perp_{\infty}} \mathcal C := \{ M \in \rmod R \mid \Ext iRMC = 0 \hbox { for all } i > 0 \hbox{ and all } C \in \mathcal C \}$$ for a class of modules $\mathcal C$, and 
\item for all $A, B \in \mathcal A$ such that $A$ is a submodule of $B$, we have $A \preceq _e B$, iff $B/A \in \mathcal A$.    
\end{itemize}
Notice that $\mathcal A = {}^{\perp_{\infty}} \mathcal C$ is closed under (arbitrary) direct sums and direct summands, and $\preceq _e$ refines direct summands.
\end{definition}

Let $\mathcal A$ be a class of modules. We will denote by $\varinjlim \mathcal A$ ( $\varinjlim _\omega \mathcal A$ ) the class of all modules $M$ such that $M$ is the direct limit of a direct system 
(a countable direct system) consisting of modules from $\mathcal A$. For example, if $\mathcal A = \mathcal P_0$ is the class of all projective modules, then $\varinjlim \mathcal A = \mathcal F _0$ is the class of all flat modules. A class of modules $\mathcal A$ is \emph{closed under (countable) direct limits} in case  $\mathcal A = \varinjlim \mathcal A$ ($\mathcal A = \varinjlim _\omega \mathcal A$). 
   
\medskip
\subsection{Deconstructible classes}

For a class $\mathcal C$ of modules, $\Filt{\mathcal C}$ will denote the class of all \emph{$\mathcal C$-filtered} modules (also called \emph{transfinite extensions} of the modules from $\mathcal C$). These are the modules $M$ that possess a \emph{$\mathcal C$-filtration}, that is, an increasing chain of submodules, $\mathcal M = ( M_\alpha \mid \alpha \leq \sigma )$, such that $M_0 = 0$, $M_\sigma = M$, $M_\alpha = \bigcup_{\beta < \alpha} M_\beta$ for each limit ordinal $\alpha \leq \sigma$, and $M_{\alpha + 1}/M_{\alpha} \cong C_\alpha$ for some $C_\alpha \in \mathcal C$ for each $\alpha < \sigma$. The ordinal $\sigma$ will be called the \emph{length} of the $\mathcal C$-filtration $\mathcal M$. If $\mathcal C = \rmod R$, we simply say that $\mathcal M$ is a~\emph{filtration}, instead of $\rmod R$-filtration.

For an infinite cardinal $\lambda$, we will also use the term \emph{$\lambda$-filtration} to denote a~filtration $\mathcal M = ( M_\alpha \mid \alpha \leq \lambda)$ where $M_\alpha$ is $<\!\lambda$-generated for each $\alpha<\lambda$. Notice that if $\lambda$ is a~regular uncountable cardinal, then any two $\lambda$-filtrations of a~given ($\lambda$-generated) module $M$ coincide on a club (= closed and unbounded subset) in $\lambda$, cf.\ \cite[II.4.12]{EM}.  
  
Let $\mathcal A \subseteq \rmod R$. If $\kappa$ is an infinite cardinal, then $\mathcal A$ is called \emph{$\kappa$-deconstructible} provided that $\mathcal A = \Filt{\mathcal A ^{< \kappa}}$ where $\mathcal A ^{< \kappa}$ denotes the class of all $< \kappa$-presented modules in $\mathcal A$. The class $\mathcal A$ is \emph{deconstructible} provided it is $\kappa$-deconstructible for some infinite cardinal $\kappa$. Each deconstructible class of modules $\mathcal A$ is closed under transfinite extensions, that is, it satisfies $\mathcal A = \Filt{\mathcal A}$, so in particular, $\mathcal A$ is closed under extensions and arbitrary direct sums. 

While most classes closed under transfinite extensions occurring in homological algebra are deconstructible, this is not always the case. 
The class $\mathcal F \mathcal M$ of all flat Mittag-Leffler (= $\aleph_1$-projective) modules over any non-right perfect ring satisfies $\mathcal F \mathcal M = \Filt{\mathcal F \mathcal M}$, but it is not deconstructible (see Example \ref{FM} below). Further, the deconstructibility of the class of all Whitehead groups is independent of ZFC+ GCH by \cite{ES}. However, Cox has recently proved consistency of deconstructibility of all left hand classes of cotorsion pairs over any right hereditary ring \cite{C}.    

In \cite{BET}, the following necessary and sufficient condition on a left hand class of a hereditary cotorsion pair to fit in an AEC of the roots of Ext was proved:

\begin{theorem}\label{roots} \cite{BET} (see also \cite[Theorem 10.27]{GT}). Let $\mathcal A  = {}^{\perp_{\infty}} \mathcal C$ for a class of modules $\mathcal C$. Then $(\mathcal A, \preceq _e)$ is an AEC, iff the class $\mathcal A$ is deconstructible and closed under direct limits. 

Moreover, if $\mathcal A$ is $\kappa^+$-deconstructible, then the L\"{o}wenheim--Skolem number of $\mathcal A$ is $\leq \kappa$. 
\end{theorem} 

In Section \ref{limits}, we will show that closure under direct limits in Theorem \ref{roots} is actually a~consequence of deconstructibility. In fact, we will prove a~much more general fact: closure of an arbitrary class of modules $\mathcal A$ under direct limits follows from its deconstructibility whenever $\mathcal A$ is closed under direct summands, and there exists an AEC of the form $(\mathcal A, \preceq )$ such that the partial order $\preceq$ refines direct summands.

\medskip
\subsection{Categoricity}

Let $\lambda$ be an infinite  cardinal. Then a class of modules $\mathcal D$ is \emph{$\lambda$-categorical} (or \emph{categorical in $\lambda$}) provided that $\mathcal D$ contains a module of cardinality $\lambda$, and all modules of cardinality $\lambda$ contained in $\mathcal D$ are isomorphic. 

\medskip
A long-standing open problem of abstract model theory concerns categoricity for the general abstract elementary classes defined in \ref{defroots}, see \cite[N.4.3]{S}, \cite[Conjecture 1.3]{V}. It asks about the existence, and the relation to the L\"{o}wenheim--Skolem number, of a cardinal $\lambda$ such that $\lambda$-categoricity already implies $\lambda ^\prime$-categoricity for each $\lambda ^\prime \geq \lambda$: 

\medskip
\noindent {\bf Shelah's Categoricity Conjecture} (SCC): Let $(\mathcal A, \preceq)$ be an AEC with L\"{o}wenheim--Skolem number $\kappa$ such that $\mathcal A$ is categorical in a~cardinal $\lambda \geq \beth_{(2^\kappa)^+}$. Then $\mathcal A$ is categorical in all cardinals $\geq \beth_{(2^\kappa)^+}$. 

\medskip
Here, the function $\beth$ (beth) is defined by induction on the set of all ordinals as follows: $\beth_0 = \aleph_0$, $\beth_{\alpha + 1} = 2^{\beth_{\alpha}}$, and $\beth_\alpha = \sup_{\beta < \alpha} \beth_{\beta}$ when $\alpha$ is a limit ordinal. Note that $\kappa \leq \aleph _\kappa \leq \beth _\kappa$ for each cardinal $\kappa$.

SCC has been verified in a number of particular cases, see \cite{S}, \cite{V}, et al. For the particular case of the AEC's of roots of Ext, it was proved in \cite{T} with an improved bound of $\kappa^+$ in place of $\beth_{(2^\kappa)^+}$. 

Analogues of SCC for various settings beyond AECs have recently been proved in \cite{ABV} and \cite{M}. In Section \ref{Scategoricity}, we will prove the analogue of SCC for arbitrary deconstructible classes of modules (that is, even for those that do not fit in any AEC), with the improved bound of $(2^\kappa)^+$ in place of $\beth_{(2^\kappa)^+}$.   

For further properties of the notions defined above, we refer to \cite{EM} and \cite{GT}.

\section{Closure under direct limits}\label{limits}

In general, deconstructible classes of modules are not closed under direct summands---just consider the particular case of free modules (or Example \ref{nclosed} below). However, if $\mathcal A$ is a deconstructible class and $\mathcal B$ is the class of all modules isomorphic to direct summands of modules in $\mathcal A$, then $\mathcal B$ is also deconstructible (cf.\ \cite[\S 7.2]{GT}, which as a particular case gives the classic Kaplansky theorem on the structure of projective modules). So it is not much of a restriction that we will consider here deconstructible classes of modules closed under direct summands. 

\smallskip

The proof of our first main theorem makes use of the following result going back to Hill (see e.g.\ \cite[Theorem 7.10]{GT}):
 
\begin{lemma} \label{Hill}
Let $\kappa$ be a regular infinite cardinal. Let $\mathcal S$ be a class of $< \kappa$-presented modules and $M$ a module possessing an $\mathcal S$-filtration $(M_\alpha \mid \alpha\leq\sigma)$. Then there is a family $\mathcal F$ of submodules of $M$ such that:
\begin{enumerate}
\item $M_\alpha \in \mathcal F$ for all $\alpha\leq\sigma$.
\item $\mathcal F$ is closed under arbitrary sums and intersections.
\item For each $N,P \in \mathcal F$ such that $N \subseteq P$, the module $P/N$ is $\mathcal S$-filtered.
\item For each $N \in \mathcal F$ and a subset $X \subseteq M$ of cardinality $< \kappa$, there is $P \in \mathcal F$ such that $N \cup X \subseteq P$ and $P/N$ is $< \kappa$-presented.\qed
\end{enumerate}
\end{lemma}

Surprisingly, the key ingredient of our proof is a construction of tree modules (denoted by $L$ below) that has already proven instrumental for the solution of Auslander's problem on existence of right almost split maps in \cite{Sa}:

\begin{lemma} \label{l:3.4}\cite[Lemma~3.4]{Sa} Assume that $\mathcal C = (C_\alpha,f_{\beta\alpha}\colon C_\alpha\to C_\beta \mid \alpha\leq\beta<\mu)$ is a~well-ordered directed system of modules indexed by an infinite regular cardinal $\mu$. Let $\lambda$ be a~cardinal such that $\lambda^\mu > \lambda^{<\mu} =\lambda\geq |R|$ and that $\mathcal C$ consists of $\lambda$-presented modules. Put $N = \bigoplus_{\alpha<\mu}C_\alpha$ and $C = \varinjlim \mathcal C$. Then the following holds:
\begin{enumerate}
	\item There is a~short exact sequence $0 \longrightarrow D \overset{\subseteq}\longrightarrow L \overset{g}\longrightarrow C^{(\lambda^\mu)} \longrightarrow 0$ where $|D|\leq\lambda$.
	\item There is an epimorphism $\pi\colon N^{(\lambda^\mu)} \to L$ where $L$ is as in $(i)$. Moreover, for each \emph{finite} $S\subset \lambda^\mu$, the module $\pi(N^{(S)})$ is a~direct summand in $L$ and it decomposes as the direct sum of modules isomorphic to modules in $\{C_\alpha\mid \alpha<\mu\}$. \hfill\qed
\end{enumerate}
\end{lemma}

Our main theorem now follows. Notice the vital role the assumption that $\preceq$ refines direct summands plays in the proof.

\begin{theorem}\label{t:closedl} Let $\mathcal A=\Filt{\mathcal S}$ where $\mathcal S$ is a~set of modules. Assume that $\mathcal A$ is closed under direct summands. Let $\preceq$ be a~partial order on $\mathcal A$ such that $\preceq$ refines direct summands. Moreover, assume that conditions (A1)(i) and (A1)(iii) hold for the pair $(\mathcal A, \preceq)$. Then the class $\mathcal A$ is closed under direct limits.
\end{theorem}

\begin{proof} It is enough to show that $\mathcal A$ is closed under direct limits of well-ordered directed systems of the form $\mathcal C = (C_\alpha,f_{\beta\alpha}\colon C_\alpha\to C_\beta \mid \alpha\leq\beta<\mu)$ where $\mu$ is an infinite regular cardinal. Assume that we are given any such directed system $\mathcal C$ with $C_\alpha\in\mathcal A$ for each $\alpha<\mu$.

We use Lemma~\ref{l:3.4} for this $\mathcal C$ and for any~$\lambda$ satisfying the assumption of the lemma and such that all modules in $\mathcal S$ are $\lambda$-presented; any large enough strong limit singular cardinal $\lambda$ with cofinality $\mu$ suffices here, cf.\ \cite[Theorem~5.20(iii)]{J}. We claim that the module $L$ from the lemma belongs to $\mathcal A$.

To show it, let us first denote, for each $S\subseteq \lambda^\mu$, by $L_S$ the module $\pi(N^{(S)})$. In particular, $L = L_{\lambda^\mu}$. We prove that $L_S\in\mathcal A$ by induction on the cardinality of $S$. At the same time, we show that, for each $S^\prime\subseteq S$, $L_{S^\prime}\preceq L_S$.

If $S$ is finite, then $L_S\in\mathcal A$ by part $(2)$ from Lemma~\ref{l:3.4}, and $L_{S^\prime}\preceq L_S$ holds since $\preceq$ refines direct summands and $L_{S^\prime}$ is a~direct summand of $L_S$.

Assume now that $S\subseteq \lambda^\mu$ is infinite and that $L_U\preceq L_V$ holds whenever $U\subseteq V$ and $|V|<|S|$. We write $S$ as the union of a~continuous increasing chain $(S_\alpha\mid \alpha<\cf{|S|})$ of subsets of $S$ of cardinality strictly less than $|S|$. Since $L_{S_\alpha}\preceq L_{S_{\alpha+1}}$ holds for each $\alpha$ by the inductive assumption, we get that $L_S\in\mathcal A$ by (A1)(i).

For the sake of contradiction, suppose that there is $S^\prime\subseteq S$ such that $L_{S^\prime}\npreceq L_S$. Assume, moreover, that $S^\prime$ has the least possible cardinality. Then $S^\prime$ is infinite since otherwise $L_{S^\prime}$ is a~direct summand in $L_S$, and so $L_{S^\prime}\preceq L_S$. Again, we write $S^\prime$ as the union of a~continuous increasing chain $(S^\prime_\alpha\mid \alpha<\cf{|S^\prime|})$ of subsets of $S^\prime$ of cardinality strictly less than $|S^\prime|$. By the inductive assumption, this chain is $\preceq$-increasing. Using the minimality of $|S^\prime|$, we have $L_{S^\prime_\alpha}\preceq L_S$ for each $\alpha$. It follows that $L_{S^\prime}\preceq L_S$ holds true by (A1)(iii).

We have shown that $L\in\mathcal A$. Now put $\kappa = \lambda^+$ and fix a~filtration of $L$ consisting of modules isomorphic to the ones from $\mathcal S$. Using Lemma~\ref{Hill} for our $\mathcal S$ and $\kappa$, we obtain a~submodule $P$ of $L$ such that $D\subseteq P$, $|P|\leq\lambda$ and $L/P\in\mathcal A$. In particular, $L/P\cong \frac{L/D}{P/D} \cong C^{(\lambda^\mu)}/g(P)\in\mathcal A$. However, $|g(P)|\leq\lambda<\lambda^\mu$ implies that there is an $S\subset\lambda^\mu$ of cardinality $\lambda$ such that $C^{(\lambda^\mu)}/g(P)\cong C^{(S)}/g(P)\oplus C^{(\lambda^\mu\setminus S)}\in\mathcal A$. It follows that $C = \varinjlim\mathcal C\in\mathcal A$ since $\mathcal A$ is closed under direct summands.
\end{proof}

\begin{remark} Notice that the assumption that $\preceq$ refines direct summands is necessary in Theorem \ref{t:closedl}: If $\mathcal A$ is any deconstructible class of modules closed under direct summands, but not closed under direct limits (such as the class of all projective modules over a non-right perfect ring), then $(\mathcal A,\preceq)$ trivially satisfies conditions (A1) and (A2) in the case when $A \preceq B$ stands for $A = B$. However, if we assume that $(\mathcal A, \preceq)$ is an AEC, then the assumption that $\preceq$ refines direct summands is not needed (at least in some models of ZFC), as is shown in the Appendix.

Since each class of modules closed under countable direct limits is necessarily closed under direct summands, also the assumption of $\mathcal A$ being closed under direct summands is necessary in Theorem \ref{t:closedl}. This assumption may fail even if $\mathcal A$ fits in an AEC as shown by the following example.
\end{remark}

\begin{example}\label{nclosed} Let $S$ be a ring, and $(\mathcal A,\preceq)$ be any AEC of right $S$-modules such that $\mathcal A$ is deconstructible (e.g., any of the AECs from Theorem \ref{roots} will do). 

Let $R = S \boxplus S$ be the ring direct product of two copies of $S$. Then $R = eR \oplus fR$, where $e = (1,0)$ and $f = (0,1)$ are orthogonal central idempotents of $R$, and there are ring isomorphisms $eR \cong S \cong fR$. 

Let $\mathcal A ^\prime = \{ N \in \rmod R \mid \, \exists M \in \mathcal A : \, M = Ne \,\&\,  M = Nf \}$. That is, $\mathcal A ^\prime$ consists of the pairs $(M,M)$ where $M \in \mathcal A$. 

For $N, P \in \mathcal A ^\prime$, let $N \preceq ^\prime P$, iff $Ne \preceq Pe$. Then $(\mathcal A ^\prime,\preceq ^\prime)$ is an AEC in $\rmod R$. However, $\mathcal A ^\prime$ is not closed under direct summands: if $0 \neq N \in \mathcal A ^\prime$, then $N$ decomposes as $N = Ne \oplus Nf$ in $\rmod R$, but neither of the right $R$-modules $Ne$ and $Nf$ belongs to $\mathcal A ^\prime$.      
\end{example}

Using Lemma \ref{Hill}, it is easy to see that each deconstructible class of modules $\mathcal A$ is Kaplansky. Here, $\mathcal A$ is a \emph{Kaplansky class} in case $\mathcal A$ is \emph{$\kappa$-Kaplansky} for some infinite cardinal $\kappa$. The latter means that for each $A \in \mathcal A$ and each subset $X$ of $A$ of cardinality $\leq \kappa$, there exists a $\leq \kappa$-presented module $B \in \mathcal A$ such that $X \subseteq B \subseteq A$ and $B/A \in \mathcal A$. Conversely, a~Kaplansky class $\mathcal A$ \emph{closed under direct limits} is deconstructible (see \cite[Theorem 10.3]{GT}), but not in general.

In the last paragraph of the proof of Theorem~\ref{t:closedl}, we used Lemma~\ref{Hill} to obtain a~suitable submodule $P$ of $L$. This could have been alternatively done assuming that $\mathcal A = \Filt{\mathcal A}$ and $\mathcal A$ be $\lambda$-Kaplansky instead of $\mathcal A = \Filt{\mathcal S}$. The following example, however, shows that the variant of Theorem \ref{t:closedl} where {\lq}deconstructible{\rq} is replaced by the weaker term {\lq}Kaplansky{\rq}, fails in general. 

\begin{example}\label{FM} Let $R$ be a non-right perfect ring and $\mathcal F \mathcal M$ be the class of all flat Mittag-Leffler (= $\aleph_1$-projective) modules. Then $\mathcal F \mathcal M$ is $\kappa$-Kaplansky for $\kappa = 2^{|R|}$ \cite{SaT}, but $\mathcal F \mathcal M$ is not deconstructible: for each $\lambda \geq \aleph_0$, there exists $M_\lambda \in \mathcal F \mathcal M$ such that $M_\lambda$ is $\lambda^+$-presented, but $M_\lambda \notin \Filt {\mathcal F \mathcal M ^{< \lambda ^+}}$. A proof of the latter fact appears in \cite{HT} (see also \cite[\S 10.2]{GT}).
 
As pointed out by Mazari-Armida, $(\mathcal F \mathcal M,\subseteq^*)$ is an AEC, where $\subseteq^*$ denotes the relation of being a pure submodule. Indeed, condition (A1)(i) holds because the modules $M \in \mathcal F \mathcal M$ are characterized by the property that each finite subset of $M$ is contained in a countably generated pure and projective submodule of $M$, \cite[Corollary 3.19]{GT}. Conditions (A1)(ii)--(iii), (A2) and (A3) follow from well-known properties of purity in $\rmod R$. 

So $(\mathcal F \mathcal M,\subseteq^*)$ is an AEC, $\mathcal F \mathcal M$ is a Kaplansky class (even $\aleph_1$-Kaplansky when $|R| = \aleph_0$ and CH holds). However, the class $(\mathcal F \mathcal M) ^{< \aleph_1} = \mathcal P _0 ^{< \aleph_1}$ is not closed under countable direct limits. Indeed, since $R$ is not right perfect, there exist countably presented flat modules that are not projective. 
\end{example}

\begin{remark} $(\mathcal F \mathcal M,\subseteq^*)$ is an AEC, but $\mathcal F \mathcal M$ is not closed under direct limits when $R$ is not right perfect. So the proof of Theorem \ref{t:closedl} in this particular setting yields an alternative proof of non-deconstructibility of the class $\mathcal F \mathcal M$ to the proof presented in \cite{HT}.
\end{remark}  

Finally, let us note that under the assumption of Vop\v{e}nka’s Principle, the converse of Theorem~\ref{t:closedl} holds in the following form: Each class of modules closed under finite direct sums and direct limits is deconstructible (cf.\ Remark~4 and Corollary~3.4 in \cite{PPT}).

\section{Shelah's categoricity conjecture for deconstructible classes of modules}\label{Scategoricity}

The goal of this section is to prove a version of Shelah's Categoricity Conjecture for all deconstructible classes of modules, with an improved bound of $(2^\kappa)^+$ in place of $\beth_{(2^\kappa)^+}$. That is, we will prove the following

\begin{theorem}\label{maindec} Let $R$ be a ring and $\kappa \geq \card R + \aleph_0$. Let $\mathcal D$ be a $\kappa^+$-deconstructible class of modules. Assume that $\mathcal D$ is $\lambda$-categorical for some $\lambda \geq (2^\kappa)^+$. Then $\mathcal D$ is $\lambda$-categorical for each $\lambda \geq (2^\kappa)^+$.
\end{theorem}

In order to prove the theorem, we will need several auxiliary results. The first is a classic fact concerning direct sum decompositions of modules due to Carol Walker \cite[Theorem~4.2]{W}.

\begin{lemma}\label{Walker} Let $R$ be a ring, $\kappa$ be an infinite cardinal, and $M$ and $N$ be modules such that $N$ is a direct summand in $M$. Assume that $M$ is a direct sum of $\kappa$-generated modules. Then so is $N$.
\end{lemma}

The proof of the second fact makes use of a version of Eilenberg's trick \cite[Proposition~18.1]{P} adapted to our setting, and of Lemma~\ref{Hill}.

\begin{lemma}\label{Eilenberg} Let $R$ be a ring and $\mathcal D$ a class of modules closed under direct sums which is $\lambda$-categorical for some $\lambda \geq \aleph_0$. Let $0 \neq M, N \in \mathcal D$ and let $\mu = \card M + \aleph_0$ and $\nu = \card N + \aleph_0$. Assume that $\mu \leq \nu \leq \lambda$. Then $M^{(\nu)} \cong N^{(\nu)}$. Moreover, if $\mathcal D$ is $\mu^+$-deconstructible, then $N$ is a~direct sum of $\mu$-presented modules from $\mathcal D$.
\end{lemma}  

\begin{proof} The $\lambda$-categoricity of $\mathcal D$ implies that $M^{(\lambda)} \cong N^{(\lambda)}$. By our cardinality assumptions it follows that $M^{(\nu)}$ is a direct summand in $N^{(\nu)}$, so $M^{(\nu)} \oplus X \cong N^{(\nu)}$ for some  $X \in \rmod R$. Similarly $N^{(\nu)} \oplus Y \cong M^{(\nu)}$ for some $Y \in \rmod R$. 

Then $N^{(\nu)} \cong  (N^{(\nu)})^{(\omega)} \cong M^{(\nu)} \oplus (X \oplus M^{(\nu)})^{(\omega)} \cong M^{(\nu)} \oplus (N^{(\nu)})^{(\omega)} \cong  M^{(\nu)} \oplus N^{(\nu)}$. Similarly  $M^{(\nu)} \cong M^{(\nu)} \oplus N^{(\nu)}$. 

\smallskip

It remains to prove the moreover clause. Fix a~generating set $\{n_\alpha\mid \alpha<\xi\}$ of $N$ where $\xi$ is an ordinal number. By Lemma~\ref{Walker}, it follows that $N = \bigoplus_{i\in I} M_i$ where $M_i$ is non-zero and $\mu$-generated for every $i\in I$. Since $\mathcal D$ is $\mu^+$-deconstructible, we can fix a~$\mathcal D$-filtration $(N_\alpha\mid \alpha\leq\sigma)$ of $N$ with $\mu$-presented consecutive factors and use Lemma~\ref{Hill} to obtain a~family $\mathcal F$ for this filtration and $\kappa := \mu^+$.

To prove the moreover clause, we are going to build another $\mathcal D^{<\kappa}$-filtration, $(F_\alpha\in\mathcal F \mid \alpha\leq\xi)$, of $N$ which will consist of canonical direct summands of $\bigoplus_{i\in I} M_i$. Put $F_0 = 0$ and $I_0 = \varnothing$. Assume that $0<\alpha\leq\xi$ and $F_\beta = \bigoplus_{i\in I_\beta} M_i\in \mathcal F$ is defined for each $\beta<\alpha\leq\xi$ in such a~way that $F_{\beta+1}/F_\beta \in \mathcal D^{<\kappa}$ if $\beta+1<\alpha$. If $\alpha$ is a~limit ordinal, we put $I_\alpha = \bigcup_{\beta<\alpha} I_\beta$ and $F_\alpha = \bigcup_{\beta<\alpha} F_\beta$. The latter module is in $\mathcal F$ by Lemma~\ref{Hill}(ii), and it is clear that $F_\alpha = \bigoplus_{i\in I_\alpha} M_i$.

Now let $\alpha = \gamma+1$ and put $J_0 = I_\gamma$. Using Lemma~\ref{Hill}(iv), we find a~$P_0\in\mathcal F$ such that $F_\gamma\cup\{n_\gamma\}\subseteq P_0$ and $P_0/F_\gamma$ is $\mu$-presented. It follows that there is a~$J_1\subseteq I$ such that $J_0\subseteq J_1$, $|J_1\setminus J_0|\leq\mu$ and $P_0\subseteq \bigoplus_{i\in J_1} M_i$.

Since $\bigoplus_{i\in J_1\setminus J_0} M_i$ is $\mu$-generated, we can use Lemma~\ref{Hill}(iv) again and obtain a~$P_1\in\mathcal F$ such that $\bigoplus_{i\in J_1} M_i\subseteq P_1$ and $P_1/P_0$ is $\mu$-presented. We continue by choosing $J_2\subseteq I$ such that $J_1\subseteq J_2$, $|J_2\setminus J_1|\leq\mu$ and $P_1\subseteq \bigoplus_{i\in J_2} M_i$, and so on.

After countably many steps, we define $I_\alpha = \bigcup_{m<\omega} J_m$ and $F_\alpha = \bigcup_{m<\omega} P_m$. It immediately follows that $F_\alpha = \bigoplus_{i\in I_\alpha} M_i$ and $F_\alpha\in\mathcal F$. Moreover, we get $F_\alpha/F_\gamma\in\mathcal D$ by Lemma~\ref{Hill}(iii). Finally, $F_\alpha/F_\gamma$ is $\mu$-presented since it is the union of the countable ascending chain $0\subseteq P_0/F_\gamma\subseteq P_1/F_\gamma\subseteq\dots$ whose consecutive factors are $\mu$-presented by the construction.

Finally, $F_\xi = N$ since $\{n_\alpha\mid \alpha<\xi\}\subseteq F_\xi$. It follows that $N = \bigoplus_{\alpha<\xi} D_\alpha$ where $D_\alpha = \bigoplus_{i\in I_{\alpha+1}\setminus I_\alpha} M_i \cong F_{\alpha+1}/F_\alpha$, whence $N$ is a~direct sum of $\mu$-presented modules from~$\mathcal D$.
\end{proof}

Recall that a module $M$ is a \emph{strong splitter} provided that $\Ext 1RM{M^{(I)}} = 0$ for each set $I$. Notice that if $M$ is a strong splitter, then so is any direct summand of $M^{(I)}$ for any set $I$. Also, if $M$ is a $\kappa$-presented module and $\kappa \geq \aleph_0$, then $M$ is a strong splitter, iff $\Ext 1RM{M^{(\kappa)}} = 0$ (see e.g.\ \cite[Lemma 2.2]{T}).

The following result was proved in \cite[Theorem 2.12]{T}.  

\begin{lemma}\label{regular} Let $R$ be a ring and $\mathcal D$ a deconstructible class of modules. Assume that $\mathcal D$ contains a  module $M$ which is not a strong splitter. Let $\nu = \card M + \card R + \aleph_0$. Then $\mathcal D$ is not $\lambda$-categorical for any regular cardinal $\lambda \geq \nu^+$.
\end{lemma}  

The proof of Theorem \ref{maindec} will proceed by considering two alternatives for the class $\mathcal D$ depending on the existence or non-existence of non-zero strong splitters of cardinality $\leq \kappa$ in $\mathcal D$. 
The first alternative generalizes the basic case of projective modules considered in \cite[Proposition 1.1]{T}: 

\begin{lemma}\label{splitters} Let $R$ be a ring and $\mathcal D$ be a $\kappa^+$-deconstructible class of modules for some $\kappa \geq \card R + \aleph_0$. Let $\mathcal C$ be a representative set of all $\kappa$-generated modules in $\mathcal D$ (whence $\mathcal D = \Filt{\mathcal C}$ and $\card{\mathcal C} \leq 2^\kappa$). Let $M_{\mathcal C} = \bigoplus_{C \in \mathcal C} C$ (so $M_{\mathcal C} \in \mathcal D$ and $\card {M_{\mathcal C}} \leq 2^\kappa$). 

Assume that $\mathcal C$ contains a non-zero strong splitter. Then the following conditions are equivalent

\begin{enumerate}
\item $\mathcal D$ is $\lambda$-categorical for each $\lambda > 2^\kappa$,
\item $\mathcal D$ is $\lambda$-categorical for some $\lambda \geq 2^\kappa$,
\item $M_{\mathcal C}^{(2^\kappa)} \cong P^{(2^\kappa)}$ for each $0 \neq P \in \mathcal C$.
\end{enumerate}
\end{lemma}

\begin{proof} That (i) $\implies$ (ii) is obvious. If (ii) holds, then (iii) follows by Lemma \ref{Eilenberg}, in the setting when $M = P$ and $N = M_{\mathcal C}$.


Assume (iii). First, we choose a non-zero strong splitter $P \in \mathcal C$. Since $M_{\mathcal C}$ is a direct summand of the strong splitter $P^{(2^\kappa)}$, also $M_{\mathcal C}$ is a strong splitter. Thus we can proceed as in Alternative 1 in \cite[p.\,378]{T}: Since $M_{\mathcal C}$ is a strong splitter, by induction on the length of a $\mathcal C$-filtration of a module $D \in \mathcal D$, we infer that each module $D \in \mathcal D$ is isomorphic to a~direct sum of modules from $\mathcal C$. 

Let $D$ be an arbitrary module of cardinality $\lambda > 2^\kappa$ in $\mathcal D$. By the above, $D \cong \bigoplus_{C \in \mathcal C\setminus\{0\}} C^{(\kappa_C)}$ for some cardinals $\kappa_C \geq 0$ ($C \in \mathcal C$). Let $D^\prime = \bigoplus_{C \in \mathcal C ^\prime} C^{(\kappa_C)}$ where $\mathcal C ^\prime$ is the set of all $C^\prime \in \mathcal C$ with $\kappa_C < 2^\kappa$. By (iii), for each $C \in \mathcal C \setminus \mathcal C ^\prime$, $C^{(\kappa_C)} \cong P^{(\kappa_C)}$. Since $\card D = \lambda > \card {\mathcal C} \geq \card {\mathcal C^\prime}$, we infer that $\card {D ^\prime} < \lambda$, and $D \cong D^\prime \oplus \bigoplus_{C \in \mathcal C \setminus \mathcal C ^\prime} P^{(\kappa_C)}$. Since $\bigoplus_{C \in \mathcal C \setminus \mathcal C ^\prime} P^{(\kappa_C)}$ has cardinality $\lambda$, it is isomorphic to $P^{(\lambda)}$, whence $D \cong D^\prime \oplus P^{(\lambda)}$. By (iii), $D^\prime$ is isomorphic to a direct summand in $P^{(\lambda)}$, so $D^\prime \oplus X \cong P^{(\lambda)}$. 

We finish the proof of condition (i) by adapting Eilenberg's Trick to our setting as follows. First, we have $D \oplus X \cong P^{(\lambda)} \oplus P^{(\lambda)} \cong P^{(\lambda)}$. So $P^{(\lambda)} \cong (P^{(\lambda)})^{(\omega)} \cong D \oplus (X \oplus D)^{(\omega)} \cong D \oplus P^{(\lambda)}$. Thus $P^{(\lambda)} \cong D \oplus P^{(\lambda)} \cong D^\prime \oplus P^{(\lambda)} \oplus P^{(\lambda)} \cong D^\prime \oplus P^{(\lambda)} \cong D$. We conclude that $D \cong P^{(\lambda)}$, whence (i) holds.
\end{proof}  

\begin{lemma}\label{nsplitters} Let $R$ be a ring and $\mathcal D$ be a $\kappa^+$-deconstructible class of modules for some $\kappa \geq \card R + \aleph_0$. Assume that $\mathcal D$ contains no non-zero strong splitters of cardinality $\leq \kappa$. Then $\mathcal D$ is not $\lambda$-categorical for any $\lambda > 2^\kappa$.
\end{lemma}   

\begin{proof} We have $\mathcal D = \Filt{\mathcal C}$ where $\mathcal C$ denotes a representative set of all modules $M \in \mathcal D$ of cardinality $\leq \kappa$. By assumption, $\mathcal C$ contains no non-zero strong splitters. So by Lemma \ref{regular}, $\mathcal D$ is not $\lambda$-categorical for any regular cardinal $\lambda \geq \kappa^+$. In particular, $\mathcal D$ is not $\lambda$-categorical for $\lambda = (2^\kappa)^+$.

Assume that $\mathcal D$ is $\lambda$-categorical for a cardinal $\lambda > (2^\kappa)^+$. Lemma \ref{Eilenberg} for $\mu = \kappa$ implies that each module $D \in \mathcal D$ with $\kappa \leq \card D \leq \lambda$ is isomorphic to a direct sum of modules from~$\mathcal C$. Let $M$ be any non-zero module from $\mathcal C$. 

By induction on $\alpha < \lambda$, we will construct a strictly increasing chain of modules $\mathcal M = ( M_\alpha\in\mathcal D^{<\lambda} \mid \alpha \leq (2^\kappa)^{++} )$ as follows: $M_0 = 0$, and $M_\alpha = \bigcup_{\beta < \alpha} M_\beta$ for each limit ordinal $\alpha \leq (2^\kappa)^{++}$. Further, we let $M_{\alpha + 1} = M_\alpha \oplus M$ for each $\alpha < (2^\kappa)^+$. 

If $(2^\kappa)^+ \leq \alpha < \lambda$ and $M_\alpha$ is defined, then by the above, $M_\alpha$ is isomorphic to a direct sum of modules from $\mathcal C$. As $\card {\mathcal C} \leq 2^\kappa$ and $\card {M_\alpha} > 2^\kappa$, there exists $0 \neq C \in \mathcal C$ such that $C^{(\kappa)}$ is isomorphic to a direct summand in $M_\alpha$, that is, $M_\alpha \cong C^{(\kappa)} \oplus D$ for a module $D \in \mathcal D$. Since $C$ is not a strong splitter, there is a non-split short exact sequence $0 \to C^{(\kappa)} \hookrightarrow E \to C \to 0$ in $\mathcal C$, and hence a non-split short exact sequence $0 \to M_\alpha \hookrightarrow M_{\alpha + 1} \to C \to 0$ in $\mathcal D$, where $M_{\alpha + 1} \cong E \oplus D$.

Put $N := M_{(2^\kappa)^{++}}$. Then $N \in \mathcal D$, $\card N = (2^\kappa)^{++}$, so by Lemma \ref{Eilenberg}, $N$ is a direct sum of copies of modules from $\mathcal C$, say $N = \bigoplus_{\alpha < (2^\kappa)^{++}} N_\alpha$. For each $\alpha \leq (2^\kappa)^{++}$, let $P_\alpha = \bigoplus_{\beta < \alpha} N_\beta$. Then $\mathcal M$ and $\mathcal P = ( P_\alpha \mid \alpha \leq (2^\kappa)^{++} )$ are two $(2^\kappa)^{++}$-filtrations of the module $N$. So there is a club $A$ in $(2^\kappa)^{++}$ such that for each $\alpha \in A$, $M_\alpha = P_\alpha$. In particular, $A$ contains an ordinal $\alpha^\prime \geq (2^\kappa)^{+}$. Then $P_{\alpha^\prime}$ is a direct summand in $N$, and hence in $M_{\alpha^\prime + 1}$, but $M_{\alpha^\prime} = P_{\alpha^\prime}$ is not a direct summand in $M_{\alpha^\prime + 1}$ by our construction of the chain $\mathcal M$, a~contradiction.                
\end{proof} 

Now we can present a proof of Theorem \ref{maindec}:

\begin{proof} Assume that $\mathcal D$ is $\lambda$-categorical for some $\lambda \geq (2^\kappa)^+$. By Lemma \ref{nsplitters}, $\mathcal D$ contains a non-zero strong splitter $M$ of cardinality $\leq \kappa$, and by the assumption, condition (ii) of Lemma \ref{splitters} is satisfied. Thus, also condition (i) of that lemma holds true, q.e.d.   
\end{proof}

\appendix
\section{A variant of Theorem~\ref{t:closedl}}
\label{sec:appendix}

To show that the assumption that $\preceq$ refines direct summands is often redundant in Theorem~\ref{t:closedl}, we first prove a~general result concerning arbitrary AECs, not necessarily of modules. In what follows, $\mathcal L$ denotes a~first-order language.

An \emph{AEC of $\mathcal L$-structures} is a~pair $\mathbb A = (\mathcal A,\preceq)$, where $\mathcal A$ is a~class of $\mathcal L$-structures, satisfying axioms (A1)--(A3) from Definition~\ref{defroots} with `(sub)module' replaced by `(sub)structure'. The condition that $\lambda\geq |R|$ in (A3) is replaced by $\lambda\geq |\mathcal L|$.

Let $\kappa$ be a~regular infinite cardinal and $M$ an $\mathcal L$-structure. We say that a~set $\mathcal T$ consisting of $\mathcal L$-substructures of $M$ is \emph{$\kappa$-dense in $M$} provided that:
\begin{enumerate}
	\item Each $N\in\mathcal T$ has cardinality $<\!\kappa$.
	\item Each subset $X$ of (the universe of) $M$ with $|X|<\kappa$ is contained in an $N\in\mathcal T$.
	\item $\mathcal T$ is closed under union of $\subseteq$-increasing chains of length $<\!\kappa$.
\end{enumerate}

In what follows, we do not notationally distinguish $\mathcal L$-structures and their respective universes. The next result is proved similarly as \cite[Lemma~6.2]{SS2}. The main tool in its proof is borrowed from the Shelah's proof that an abelian group is strongly $\kappa$-free provided that it is $\kappa^+$-free.

\begin{proposition} \label{p:dense} Let $\mathbb A = (\mathcal A, \preceq)$ be an AEC of $\mathcal L$-structures. Let $\lambda > \LS(\mathbb A)$ be regular. Assume that $A$ is an $\mathcal L$-structure which possesses a~$\lambda^+$-dense subset consisting of elements from $\mathcal A$. Then $A\in\mathcal A$.
\end{proposition}

\begin{proof} Let us denote the $\lambda^+$-dense set by $\mathcal T$. We can assume, without loss of generality, that $|A|>\lambda$; otherwise $A\in\mathcal A$ using the $\lambda^+$-density of $\mathcal T$. Let us also denote by $\mathcal S$ the set of $\mathcal L$-substructures of $A$ of cardinality $<\!\lambda$. Since $\lambda$ is regular and greater than the cardinality of $\mathcal L$, the set $\mathcal S$ is nonempty and $\lambda$-dense in $A$.

We define the following game $\mathfrak G$: It is played in turns by Player~I and Player~II. Player I starts and chooses successively subsets $X_0, X_1, \dots$ of the universe of $A$, each of cardinality $<\!\lambda$. Player II, on each $X_n$, replies with some $N_n\in\mathcal S$. At most $\omega$ turns are played; after the first $n+1$ turns, we will have the following sequence:
$$X_0, N_0, X_1, N_1, \dotsc , X_n, N_n.$$
\noindent
Player II wins if he manages to play, for each $n\in\omega$, so that $X_n\subseteq N_n\in\mathcal A$ and $N_{n-1}\preceq N_n$ (if $n>0$). Otherwise, Player I immediately wins. For each $N\in\mathcal S\cap\mathcal A$, we also consider the variant of $\mathfrak G$ denoted by $\mathfrak G_N$ in which Player II is additionally obliged to pick $N_0$ so that $N\preceq N_0$.

First, we prove that Player I possesses no winning strategy in $\mathfrak G$. For each $K\in\mathcal T$, we fix an $\subseteq$-increasing chain $(K^\alpha\mid |K^\alpha|<\lambda, \alpha<\lambda)$ with $\bigcup_{\alpha<\lambda} K^\alpha =$~$K$. Let $s$ be a~strategy for Player I in $\mathfrak G$, i.e., a~function that gives the first move~$X_0$, and it decides what the answer should be to the play by Player II; so $X_n = s(N_0, N_1, \dots , N_{n-1})$ for $n>0$. We want to beat the strategy $s$.

Using the properties of $\mathcal T$, we easily construct $\subseteq$-increasing chains $(M_\alpha \in \mathcal S\mid \alpha<\lambda)$ and $(K_\alpha\in\mathcal T\mid \alpha<\lambda)$ such that the following holds:

\begin{enumerate}[label={(\arabic*)}]
	\setcounter{enumi}{-1}
	\item $X_0\subseteq M_0$;
	\item $M_\alpha = \bigcup _{\beta <\alpha} M_\beta$ for $\alpha <\lambda$ limit;
	\item $M_\alpha\subseteq K_\alpha$ for each $\alpha<\lambda$;
	\item $M_{\alpha + 1}\supsetneq M_\alpha\cup \bigcup _{\beta\leq\alpha}K^\alpha_\beta$ for each $\alpha<\lambda$;
	\item for each $\alpha <\lambda$, $s(M_{\alpha _0}, M_{\alpha _1}, \dotsc , M_{\alpha _n})\subseteq M_{\alpha + 1}$, whenever $n\in\omega$, $\alpha _n\leq \alpha$ and $M_{\alpha _0}\preceq M_{\alpha _1}\preceq \dotsb \preceq M_{\alpha _n}$ is played by Player II according to the rules (against the strategy $s$).
\end{enumerate}

Put $M = \bigcup _{\alpha <\lambda} M_\alpha$. We have $|M| = \lambda$ and $M = \bigcup_{\alpha<\lambda} K_\alpha\in\mathcal T$ by $(2)$ and $(3)$. In particular, $M\in\mathcal A$, and hence the set $U = \{\beta<\lambda \mid M_\beta\preceq M\}$ is closed and unbounded (recall that $\lambda>\LS(\mathbb A)$; we use (A3) and (A1)(iii) here together with $(1)$ above). Notice that, for $\beta<\gamma$ from $U$, we have $M_\beta\preceq M_\gamma$ as a~consequence of the axiom (A2) for AECs. By (4), Player II is going to beat the strategy~$s$, if he chooses the elements $N_n$ as the appropriate $M_\beta$ for $\beta\in U$.

Let us denote by $\mathcal D$ the subset of $\mathcal S$ consisting of all the $N\in\mathcal A$ for which Player~I does not possess a~winning strategy in $\mathfrak G_N$. Notice that, for any $X_0$ played by Player~I in $\mathfrak G$, Player~II can choose $N_0\in\mathcal D$: if not, Player~I would have a~winning strategy in $\mathfrak G$ already, in contradiction with what we have just proved. Analogously, for any $X_0$ played by Player~I in $\mathfrak G_N$ where $N\in\mathcal D$, Player~II can choose $N_0\in\mathcal D$ again.

To sum up: $\mathcal D$ is nonempty and, for each $N\in\mathcal D$ and $X\subseteq A$ of cardinality $<\!\lambda$, there is $N^\prime\in\mathcal D$ such that $X\subseteq N^\prime$ and $N\preceq N^\prime$. This allows us, as in the proof of \cite[Lemma~6.7]{V2}, to easily show that, for any $N,L\in\mathcal D$, there exists $B\in\mathcal S\cap\mathcal A$ such that $N\preceq B$ and $L \preceq B$: indeed, we have $B_0\in\mathcal D$ such that $N\preceq B_0$ and $L\subseteq B_0$; then we choose $B_1\in\mathcal D$ with $L\preceq B_1$ and $B_0\subseteq B_1$; we continue recursively, picking $B_{n+1}\in\mathcal D$ with $B_{n-1}\preceq B_{n+1}$ and $B_n\subseteq B_{n+1}$; finally, we put $B = \bigcup_{n<\omega}B_n$ which has the desired properties by the axiom~(A1). Consequently, if in addition $N\subseteq L$, then $N\preceq L$ by the axiom (A2). It follows that $\mathcal D$ is $\preceq$-directed and $A = \bigcup \mathcal D$. Thus $A\in\mathcal A$, again by the axiom (A1) for AECs, cf.\ \cite[Lemma~1.2]{K}.
\end{proof}

To prove the promised alternate version of Theorem~\ref{t:closedl}, we need the following additional set-theoretic assumption (consistent with ZFC) from \cite{BS}:

\medskip

\noindent\textbf{Assumption $(*)$.} \textit{For each infinite regular cardinal $\theta$, there is a~proper class of cardinals $\kappa$ such that:}
\begin{enumerate}
	\item \textit{There exists a~non-reflecting stationary set $E\subseteq\kappa^+$ consisting of ordinals with cofinality $\theta$, and}
 	\item \textit{$\kappa^{<\theta} = \kappa$.}
\end{enumerate}


\begin{theorem}\label{t:closedl2} Assume $(*)$. Let $\mathcal A=\Filt{\mathcal S}$ where $\mathcal S$ is a~set of (right $R$-)modules. Assume that $\mathcal A$ is closed under direct summands. Let $\preceq$ be a~partial order on $\mathcal A$ such that $\mathbb A = (\mathcal A, \preceq)$ is an AEC of modules. Then the class $\mathcal A$ is closed under direct limits.
\end{theorem}

\begin{proof} It is enough to show that $\mathcal A$ is closed under direct limits of continuous well-ordered directed systems $(N_\alpha,g_{\alpha\beta}\colon N_\beta\to N_\alpha \mid \beta<\alpha<\theta)$ where $\theta$ is a~regular infinite cardinal. Let such a~directed system, with $N_\alpha\in\mathcal A$ for each $\alpha<\theta$, be given and assume that $\varinjlim_{\alpha<\theta}N_\alpha\neq 0$. Put $N = \bigoplus_{\alpha<\theta} N_\alpha$. We use \cite[Proposition~3.3]{BS} where $\kappa$ satisfies that $\kappa>\LS(\mathbb A)+|N|$ and that all modules in $\mathcal S$ are $\kappa$-presented. For $\lambda = \kappa^+$, we obtain a~stationary subset $F\subseteq \lambda$ and an~$L\in$ Mod-$R$ possessing a~filtration $\mathcal L = (L_\alpha\mid \alpha\leq\lambda)$ satisfying
\begin{enumerate}
	\item[(a)] $L_\alpha\in\mathcal A$ and $|L_\alpha|<\lambda$ for each $\alpha<\lambda$, and
	\item[(b)] for every $\delta\in F$ and $\delta<\varepsilon<\lambda$, we have $\varinjlim_{\alpha<\theta}N_\alpha\cong L_{\delta+1}/L_\delta$, and the latter module splits in $L_{\varepsilon}/L_\delta$.
\end{enumerate}

It follows that $|L| = \lambda$. Putting $N^\prime = N^{(\lambda)}$, it is actually shown in the proof of \cite[Proposition~3.3]{BS} that there exists an epimorphism $\pi\colon N^\prime\to L$ such that, for each $S\subset \lambda$ of cardinality $<\!\lambda$, $\im(\pi\restriction N^{(S)})\in\mathcal A$. If $\kappa$ is regular, this means that $L$ possesses a~$\lambda$-dense system of submodules from $\mathcal A$. If $\kappa$ is singular, there has to be a~regular cardinal $\mu<\kappa$ such that $\mu>\LS(\mathbb A)+|N|$ and we see that $L$ possesses a~$\mu^+$-dense system of submodules from $\mathcal A$. Using Proposition~\ref{p:dense}, we get $L\in\mathcal A$.

To conclude our proof, we fix any $\mathcal A^{<\lambda}$-filtration of $L$ of the form $(M_\alpha \mid \alpha\leq\lambda)$. Then $|M_\alpha|<\lambda$ for each $\alpha<\lambda$. Using (a) above, it follows that the set \[U = \{\alpha<\lambda\mid M_\alpha = L_\alpha\}\] is closed and unbounded. By the part (b), we pick any $\delta\in F\cap U$ and $\varepsilon\in U$ with $\delta<\varepsilon$, and deduce that $\varinjlim_{\alpha<\theta}N_\alpha$ is isomorphic to a~direct summand in $L_{\varepsilon}/L_\delta = M_{\varepsilon}/M_\delta \in\mathcal A$. It follows that $\varinjlim_{\alpha<\theta}N_\alpha\in\mathcal A$ since $\mathcal A$ is closed under direct summands.
\end{proof}

\begin{acknowledgment}
The authors thank Kate\v{r}ina Fukov\'{a} for valuable comments.   
\end{acknowledgment}

\end{document}